\documentclass[11pt]{amsart}
\usepackage[english]{babel}
\usepackage{amssymb,amsmath,xcolor,mathrsfs,mathabx}
\usepackage{hyperref,color}%para ir de uma referência a sua localização no texto
\usepackage{graphicx}
\usepackage{wrapfig}
\usepackage[utf8]{inputenc}
\usepackage{float}
\usepackage{wrapfig}
\usepackage[all]{xy}
\usepackage{fancyhdr}% para formatar o cabeçalho
 
\numberwithin{equation}{section}%para numerar as equações pelas seções.
 
\newtheorem{theorem}{Theorem}[section]
\newtheorem{definition}{Definition}

\newtheorem{lemma}{Lemma}

\newtheorem{claim}[theorem]{Claim}

\newtheorem{proposition}[theorem]{Proposition}

\newcommand{\real}{\mathbb{R}}

\newcommand{\Vol}{{\rm{Vol}}}

\newcommand{\Ric}{{\rm{Ric}}}

\newcommand{\Hess}{{\rm{Hess\,}}}

\newcommand{\h}{\mathcal{H}^{n-1}}
\newcommand{\hf}{\mathcal{H}_f^{n-1}}

\title[Rigidity of manifolds admitting stable solutions ]{Rigidity of manifolds admitting stable solutions of an elliptic problem }

\author{M. Batista}
\address{IM, Universidade Fe\-deral de Alagoas, Macei\'o, 
AL, CEP 57072-970, Brazil}
\email{mhbs@mat.ufal.br }

\author{J. I. Santos}
\address{Instituto Federal de Alagoas, Campus Piranhas, Av. Sergipe, Xing\'o, Piranhas, AL, 57460-000, Brazil }
\email{jissivan@gmail.com}

\keywords{Stability, Elliptic problem, Complete manifolds} 
\subjclass[2010]{53C24; 58J05}

\thanks{The first author was partially supported by FAPEAL and CNPq}

\begin{document} 

\begin{abstract}
In this paper, we study geometric rigidity of Riemannian manifolds admitting stable solutions of certain elliptic problems (stability in a variational sense), that is, under suitable hypotheses, we are able to characterize the Riemannian manifold which admits a stable solution. Furthermore, under the non-negativity of the weighted Ricci curvature, we deduce several data about the stable solution and a splitting result for the manifold.
\end{abstract}
\maketitle
\section{Introduction}
The study of metric measure spaces  has flourished in last few
years, and a much better understanding of their geometric 
structure has evolved. We emphasize that the study of metric measure
spaces and its generalized curvatures go back to Lichnerowicz 
\cite{l1, l2} and more recently by Bakry and \'Emery 
\cite{BE}, in the setting of diffusion process, and it has been an
active subject in recent years. For an overview, see for instance \cite{CGY, grigoryan, m}. 
As the approach allows us to explore the setting of metric measure spaces and this theme is very attractive, see for instance the following list \cite{mw1, mw2, gp, ww, wylie} of interesting articles about this subject,  we are led to study stable solutions of a class of elliptic problems on metric measure spaces.

Throughout the paper, we use some notions
which we introduce at this moment. We recall that a 
{\it metric measure space} is a 
Riemannian manifold $({\Sigma^n},g)$ 
endowed with a real-valued smooth function $f: \Sigma \to \mathbb{R}$ 
which is used as density in the following way:
$d\Vol_f = e^{-f}d\Vol$, where 
$\Vol$ is the Riemannian measure of $\Sigma$.

Associated to this structure we have a second order
differential operator defined by
$$\Delta_f u= e^fdiv(e^{-f}\nabla u),$$
acting on space of smooth functions. This operator is known in the
literature as {\it Drift Laplacian}.
Following  \cite{BE}, the natural generalization of
Ricci curvature is defined by
$$\Ric_f=\Ric + \Hess f,$$
which is known as {\it Bakry-\'Emery Ricci tensor} or simply by {\it weighted Ricci curvature}. 
Given $g \in C^\infty(\real)$, we consider the closed Dirichlet problem
\begin{equation}\label{eq1}
\Delta_f u + g(u)=0.
\end{equation}
A solution of this problem is a critical point of an energy functional,
which we denote by $E_f$, see the details in section 2. We say that a solution $u$ of (\ref{eq1}) is
{\it stable} if the second variation of $E_f$ at $u$ is 
non-negative on $H_c^1(\Sigma).$ Lastly, we say that a metric 
measure space is {\it $f$-parabolic} if there exists no nonnegative non-constant function which is $f$-superharmonic. Now, we are
able to introduce our results. 

\medskip
The first one is read as follows:

\begin{theorem}\label{10}
 Let $\Sigma^n$ be a complete and non-compact metric measure space without boundary. Assume that $u$ is a smooth non-constant stable solution of $(\ref{eq1})$  and the weighted Ricci curvature satisfies $$Ric_f(\nabla u,\nabla u) \geq - \frac{ H_t^2}{n-1}|\nabla u|^2,$$ on the regular part of each level set $\Sigma_t=\{u=t\}$ of $u$ and $H_t$ is the mean curvature of $\Sigma_t.$
 
 If either
\begin{enumerate}
 \item [{(i)}] $\Sigma$ is $f-$parabolic and $\nabla u\in L^{\infty}(M)$,
 \\ or
 \item [{(ii)}] the function $|\nabla u|$ satisfies
 \begin{equation}
  \int_{B_R}|\nabla u|^2d\Vol_f=o(R^2\log R)\,\,\,\,\,\text{ as } R\rightarrow+\infty,
 \end{equation}
 \end{enumerate}
then, $u$ has no critical points and $\Sigma=\real\times N$ is furnished with a warped product metric $$ds^2=dt^2+ \exp\left(-2\int_0^t\lambda(s)\, ds\right)g_N,$$ for a smooth function $\lambda$ which depends of $\nabla u$.

\end{theorem}
We would like to point out that as far as we know, the result above is new even in the Riemannian case and  so we improve \cite[Theorem 1]{farina}.
\medskip

An interesting conclusion happens whether we assume that the space is non-negatively curved. The result is the following:

\begin{theorem}\label{cor1}
Let $\Sigma^n$ be a complete and non-compact metric measure space without boundary and the weighted Ricci curvature is nonnegative. Assume that $u$ is a smooth non-constant stable solution of $(\ref{eq1})$.
 
 If either
\begin{enumerate}
 \item [{(i)}] $\Sigma$ is $f-$parabolic and $\nabla u\in L^{\infty}(M)$,
 \\ or
 \item [{(ii)}] the function $|\nabla u|$ satisfies
 \begin{equation}
  \int_{B_R}|\nabla u|^2d\Vol_f=o(R^2\log R)\,\,\,\,\,\text{ as } R\rightarrow+\infty,
 \end{equation}
 \end{enumerate}
then, $u$ has no critical point and $\Sigma=\real\times N$ is furnished with the product metric $$ds^2=dt^2+ g_N,$$
and $N$ is $f$-parabolic complete totally geodesic hypersurface and its weighted Ricci curvature is nonnegative. 
%for $n\geq 3$, and $\Sigma=\mathbb R^2$ or 
%$\real\times\mathbb S^1$, with their flat metric, if $n=2$. 
Moreover, $u$ depends only on $t$, has no critical points, $\langle \nabla f, \nabla u\rangle = k|\nabla u|$ for a constant $k$ and writing $u=y(t)$, u satisfies 
$$-y'' + ky'=g(y).$$ %for some constant $k$.
%\medskip

\noindent Furthermore, if {(ii)} is met, 
\begin{equation*}\label{21}
 \Vol_f(B_R^N)=o(R^2\log R)\hspace{1.6cm}\text{ as } R\rightarrow+\infty,
\end{equation*}
and
\begin{equation*}\label{22}
 \int_{-R}^R|y^{\prime}(t)|^2dt=o\left(\frac{R^2\log R}{\Vol_f(B_R^N)}\right)\,\,\,\,\,\text{ as } R\rightarrow+\infty.
\end{equation*}

\end{theorem}
\medskip

We would like to highlight that the Allen-Cahn equation fit at our study, to see this choose $f=0$ and $g(t)=(1-t^2)t$. This equation is very interesting and well-understand on closed manifolds. Here, under suitable assumptions, our results provide some information about it on complete non-compact manifolds without boundary.
\medskip

{\bf Outline of the paper}: In section 2  we briefly survey some concepts and equivalences that we use in the paper. In section 3 we prove some technical lemmata that will be the heart of the proof of the results. In sections 4 and 5 we give the proof of them.

\section{Background}

Throughout the paper $\Sigma$ will denote a connect metric measure space 
of dimension $n\geq 2$ without boundary. We briefly fix some notation. 
Having fixed an origin $p_0$, we set $r(x)=dist(x,p_0)$, and we write 
$B_R$ for geodesic ball centered at $p_0$. If we need to emphasize 
the set under consideration, we will add a superscript symbol,
so that, for instance, we will also write $Ric_f^\Sigma$ and $B_R^\Sigma$.

The Riemannian $n$-dimensional volume will be 
indicated with $\Vol$, and the measure with density by $d\Vol_f=e^{-f}d\Vol$. 
Furthermore, in a similar way, we write $\h$ for the induced $(n-1)$-dimensional Hausdorff
measure and $d\hf=e^{-f}d\h$ for the weighted associated measure . 
Lastly, we use the symbol $\{\Omega_j\}\uparrow 
\Sigma$ for indicate a family $\{\Omega_j\}_{j\in\mathbb N}$ of 
relativity compact, open sets with smooth boundary and satisfying
$$\Omega_j\Subset\Omega_{j+1}\Subset \Sigma,\,\,\,\,\,\,\,\Sigma=\bigcup_{j=0}^{+\infty}\Omega_j,$$
where $A\Subset B$ means $\overline{A}\subseteq B$. Such a family will be called an 
exhaustion of $\Sigma$. Hereafter, we consider 
$$g\in C^\infty(\mathbb R),$$
and a solution $u$ on $\Sigma$ of
 \begin{equation*}\label{1}
 \Delta_fu+g(u)=0 \,\,\,\,\,\,\text{ in } \Sigma.
 \end{equation*}
We recall that $u$ is characterized, on each open subset $U\Subset \Sigma$, as a
critical point of the energy functional acting on functions in the Sobolev space which has compact support on $\Sigma$, that is, $E_f: H_c^1(\Sigma)\rightarrow\mathbb R$
given by
\begin{equation*}\label{2}
 E_f(w)=\frac{1}{2}\int_\Sigma|\nabla w|^2d\Vol_f-\int_\Sigma G(w)d\Vol_f,\,\,\,\,\,\text{where } G(t)=\int_0^tg(s)ds,
\end{equation*}
with respect to compactly variation in $U$. Indeed, after a straightforward computation and using the integration by parts we get:
$$E'_f(u)h = - \int_\Sigma(\Delta_f u +g(u))h\, d\Vol_f,$$
for all $h\in C^\infty_c(U)$. In a similar way, we have that
$$E_f''(u)(h,k) = - \int_\Sigma(\Delta_f h +g'(u)h)k \, d\Vol_f,$$
for $h, k \in C^\infty_c(U)$. So, the stability operator associated to $E_f$ at $u$ is given by
\begin{equation*}
 J_fh=-\Delta_fh-g^{\prime}(u)h,\,\,\,\,\,\,\,\,\,\,h\in C_c^{\infty}(\Sigma).
\end{equation*}

\begin{definition}
A function $u$ that solves $(\ref{eq1})$ is said to be a { stable solution} if
the stability operator at $u$ is nonnegative on $C_c^{\infty}(\Sigma)$, that is, if 
\begin{equation}\label{18}
 \int_\Sigma g^{\prime}(u) h^2\,  d\Vol_f \leq\int_\Sigma|\nabla h|^2\,  d\Vol_f,\,\,\,\,\,\,\,\,\,\text{for all }h \in C_c^{\infty}(\Sigma).
\end{equation}
\end{definition}

By density, we can replace $C_c^{\infty}(\Sigma)$ in $(\ref{18})$ 
by $Lip_c(\Sigma)$. By a simple adaptation of  
\cite[Theorem 1]{FCS}, the stability of $u$ turns out to be
equivalent to the existence of a positive $w\in C^\infty(\Sigma)$ 
solving $\Delta_fw+g'(u)w=0$ in $\Sigma$.
\medskip

Let $\Omega$ be an open set on $\Sigma$ and $K$ be a 
compact set in $\Omega$. We call the pair $(K,\Omega)$ 
of a $f$-capacitor and define the 
$f$-capacity cap$_f(K,\Omega)$ by
\begin{equation*}
 \text{cap}_f(K,\Omega)=\inf_{\phi\,\in\,\mathcal{L}(K,\Omega)}\int_{\Omega}|\nabla\phi|^2d\Vol_f,
\end{equation*}
where $\mathcal{L}(K,\Omega)$ is the set of Lipschitz functions 
$\phi$ on $\Sigma$  with compact support in $\overline{\Omega}$ such that 
$0\leq\phi\leq 1$ and $\phi|_K=1$.

For an open precompact set $K\subset \Omega$, we define its $f$-capacity by
$$\text{cap}_f(K,\Omega):=\text{cap}_f(\overline{K},\Omega).$$

In case that $\Omega$ coincide with $\Sigma$, we write cap$_f(K)$ for cap$_f(K,\Omega)$.
It is obvious from the definition that the set $\mathcal{L}(K,\Omega)$ 
increases on expansion of $\Omega$ (and on shrinking of $K$). 
Therefore, the capacity cap$_f(K,\Omega)$ decreases on expanding of $\Omega$
(and on shrinking of $K$). In particular, we can prove that, 
for any exhaustion sequence $\{\mathcal{E}_k\}$
$$\text{cap}_f(K):=\lim_{k\rightarrow\infty}\text{cap}_f(K,\mathcal{E}_k).$$
Moreover, the limit is independent of the exhaustion.

In the next definition we provide an analytical concept that is related with the notion of capacity. It shall be clear at next proposition. 

\begin{definition}
A metric measure space is $f$-parabolic if there exists no non-constant nonnegative
$f$-superharmonic function $u$, that is, if $\Delta_fu\leq0$ and $u\geq 0$, then $u$ is constant.
\end{definition}

Hence, we have the following characterization of 
$f$-parabolicity. For a proof see for instance \cite{grigoryan}.
\begin{proposition}\label{13}
 Let $\Sigma$ be a complete metric measure space. Then, the following are equivalent:
 \begin{enumerate}
\item [(i)] $\Sigma$ is $f$-parabolic.
\item [(ii)]cap$_f(K)=0$ for some $($then any$)$ compact set $K\subset \Sigma$.
 \end{enumerate}
\end{proposition}

The following criterion of $f$-parabolicity is well known, for more details see for instance \cite[Proposition 3.4]{greg} or \cite[Theorem 11.14]{grigoryan}.
\begin{proposition}\label{23}
 Let $p_o$ be a fixed point in a metric measure space $\Sigma$ and let
 $$L(r)=\int_{\partial B(p_o,r)}d\Vol_f\hspace{1cm}\text{and}\hspace{1cm} V(r)=\int_{B(p_o,r)} d\Vol_f.$$
 If
 $$\int_1^{\infty}\frac{dr}{L(r)}=+\infty\hspace{1cm}\text{or}\hspace{1cm} \int_1^{\infty}\frac{rdr}{V(r)}=+\infty,$$
 then $\Sigma$ is $f$-parabolic.
\end{proposition}

\section{ Key Lemmata}
The first lemma is a simplified Picone type identity for metric measure spaces.

\begin{lemma}\label{ll1}
Let $\Sigma$ be a complete manifold without boundary and let $u\in C^3(\Sigma)$ be a 
solution of $\Delta_fu+g(u)=0$ on $\Sigma$. Let 
$w\in C^2(\Sigma)$ be a supersolution of 
$\Delta_fw+g^{\prime}(u)w\leq0$ such that $w>0$ on $\Sigma$. 
Then the following inequality holds true for 
every $h\in Lip_c(\Sigma)$,
\begin{equation}\label{3}
\int_{\Sigma}w^2\left|\nabla\left(\frac{h}{w}\right)\right|^2d\Vol_f \leq\int_{\Sigma}|\nabla h|^2d\Vol_f-\int_{\Sigma}g^{\prime}(u) h^2d\Vol_f.
\end{equation}
The  inequality is indeed an equality if $w$ solves $\Delta_fw+g^{\prime}(u)w=0$ on $\Sigma$.
\end{lemma}

\begin{proof}
Multiplying $\Delta_fw+g^{\prime}(u)w$ by the test function $\dfrac{h^2}{w}$, integrating and using integration by parts we deduce
\begin{equation*}\label{4}
-\int_{\Sigma}(\Delta_f w+g^{\prime}(u)w)\frac{h^2}{w}\, d\Vol_f=
\int_{\Sigma}\left\langle\nabla\left(\frac{h^2}{w}\right),\nabla w\right\rangle d\Vol_f -\int_{\Sigma} g^{\prime}(u)h^2d\Vol_f.
\end{equation*}
Since 
$$\left\langle\nabla\left(\frac{h^2}{w}\right),\nabla w\right\rangle=2\frac{h}{w}\langle\nabla h,\nabla 
w\rangle-\frac{h^2}{w^2}|\nabla w|^2,$$
using the identity
\begin{equation*}
 w^2\left|\nabla\left(\frac{h}{w}\right)\right|^2=|\nabla h|^2+\frac{h^2}{w^2}|\nabla 
w|^2-2\frac{h}{w}\langle\nabla w,\nabla h\rangle,
\end{equation*}
we infer that
\begin{equation*}\label{5}
\left\langle\nabla\left(\frac{h^2}{w}\right),\nabla 
w\right\rangle=|\nabla h|^2-w^2\left|\nabla\left(\frac{h}{w}\right)\right|^2.
\end{equation*}
Inserting the latter equality into the integral equation at this lemma and using our hypothesis we conclude the desired result.
\end{proof}

\begin{lemma}\label{25}
 Under the assumptions of the former Lemma, for every $h\in C^\infty(\Sigma)$ the following integral inequality holds true:

 \begin{align}\label{9}
&\int_{\Sigma}[|\nabla^2 u|^2+\Ric_f(\nabla u,\nabla u)]h^2\, d\Vol_f-\int_{\Sigma} h^2|\nabla|\nabla u||^2d\Vol_f\leq\\ \nonumber
 &\hspace{1cm} \int_{\Sigma}|\nabla h|^2|\nabla 
u|^2d\Vol_f-\int_{\Sigma}w^2\left|\nabla\left(\frac{h|\nabla u|}{w}\right)\right|^2d\Vol_f.
 \end{align} 
 The inequality is indeed an equality if $\Delta_f w+g^{\prime}(u)w=0$ on $\Sigma$.
\end{lemma}
\begin{proof}
 Recall the B$\ddot{\text{o}}$chner formula
 $$\frac{1}{2}\Delta_f|\nabla u|^2=|\nabla^2 u|^2+\langle\nabla u,\nabla(\Delta_f u)\rangle+\Ric_f(\nabla u,\nabla u),$$
 for all $u\in C^3(\Sigma)$. Since $u$ solves $-\Delta_f u=g(u)$, we get
\begin{equation*}\label{6}
\frac{1}{2}\Delta_f|\nabla u|^2=|\nabla^2 u|^2-g^{\prime}(u)|\nabla u|^2+\Ric_f(\nabla u,\nabla u).
\end{equation*}
Integrating the latter equality against the test function $h^2$ and setting $$I=\displaystyle\int_{\Sigma}(|\nabla^2u|^2+\Ric_f(\nabla u,\nabla u))h^2\,  d\Vol_f, $$ we deduce that
$$\begin{array}{cll}\label{7}
\vspace{0.1cm} I &=&\displaystyle\int_{\Sigma}g^{\prime}(u)|\nabla u|^2h^2\, d\Vol_f +\frac{1}{2}\displaystyle\int_{\Sigma}h^2\Delta_f|\nabla u|^2 d\Vol_f\\ \vspace{0.1cm}
&=&\displaystyle\int_{\Sigma}g^{\prime}(u)|\nabla u|^2h^2\, d\Vol_f-\frac{1}{2}\displaystyle\int_{\Sigma}\left\langle\nabla h^2,\nabla|\nabla u|^2\right\rangle d\Vol_f\\ 
&=&\displaystyle\int_{\Sigma}g^{\prime}(u)|\nabla u|^2h^2\, d\Vol_f-\int_{\Sigma}h\langle\nabla h,\nabla|\nabla u|^2\rangle d\Vol_f.
 \end{array}$$
Now, we plugging the test function $h|\nabla u|$ into $(\ref{3})$ in Lemma \ref{ll1} to obtain:
 \begin{align*}\label{8}
0&\leq  \int_{\Sigma}|\nabla(h|\nabla u|)|^2\, d\Vol_f-\int_{\Sigma}g^{\prime}(u)h^2|\nabla u|^2d\Vol_f
-\int_{\Sigma}w^2\left|\nabla\left(\frac{h|\nabla u|}{w}\right)\right|^2d\Vol_f\\ \nonumber
&=\int_{\Sigma}|\nabla h|^2|\nabla u|^2d\Vol_f+\int_{\Sigma}h^2|\nabla|\nabla u||^2d\Vol_f+2\int_{\Sigma}h |\nabla u|\langle\nabla h,\nabla|\nabla u|\rangle d\Vol_f \\ \nonumber
&\hspace{0.4cm}-\int_{\Sigma}g^{\prime}(u)h^2|\nabla u|^2d\Vol_f -\int_{\Sigma}w^2\left|\nabla\left(\frac{h|\nabla u|}{w}\right)\right|^2d\Vol_f.\nonumber
 \end{align*}
 Recalling that $\nabla|\nabla u|^2=2|\nabla u|\nabla|\nabla u|$ weakly on $\Sigma$ and putting the integral $I$ into the former inequality we conclude $(\ref{9})$ as desired.
\end{proof}

For the next result, we fix the following notation. Denote the level set of $u$ by $\Sigma_t = \{u=t\}$ and by $A_t$ its second fundamental form. Assuming that $p$ is a regular point of $u$, that is, $\nabla u(p)\neq 0$ we are able to prove:

\begin{lemma}[Kato's inequality for functions]\label{19} Under the above notation holds:

 $$|\nabla^2 u|^2-|\nabla|\nabla u||^2=|\nabla u|^2|A_t|^2+|\nabla^T|\nabla u||^2$$
 at $p\in \Sigma_t$, where $\nabla^T$ is the tangential gradient on the level set $\Sigma_t$.
\end{lemma}
\begin{proof}
In this proof we will omit the sub-index $t$. Fix a local orthonormal frame $\{e_i\}$ on $\Sigma$, and let $\nu=\nabla u/|\nabla u|$ be the normal vector. For every vector field 
$X\in\mathfrak{X}(\Sigma)$,
$$\nabla^2 u(\nu,X)=\frac{1}{|\nabla u|}|\nabla^2 u(\nabla u,X)=\frac{1}{2|\nabla u|}\langle\nabla|\nabla u|^2, X\rangle=\langle\nabla 
|\nabla u|,X\rangle.$$
Moreover, for a level set 
$$A=-\frac{\nabla^2 u|_{T\Sigma\times T\Sigma}}{|\nabla u|}.$$
Therefore
\begin{align*}
 |\nabla^2 u|^2&=\sum_{i,j}(\nabla^2 u(e_i,e_j))^2+2\sum_j(\nabla^2 u(\nu,e_j))^2+(\nabla^2 u(\nu,\nu))^2\\
 &=|\nabla u|^2|A|^2+2\sum_j\langle\nabla|\nabla u|,e_j\rangle^2+\langle\nabla|\nabla u|,\nu\rangle^2\\
 &=|\nabla u|^2|A|^2+|\nabla^T|\nabla u||^2+|\nabla|\nabla u||^2,
\end{align*}
and so we conclude the Lemma.
\end{proof}

\begin{lemma}\label{I1} For all $h \in C^\infty(\Sigma)$ holds:
\begin{equation*}
\int_\Sigma\left(|\nabla^2 u|^2-|\nabla|\nabla u||^2 + Ric_f(\nabla u, \nabla u)\right)h^2d\Vol_f  
\leq 2\int_\Sigma|\nabla h|^2|\nabla u|^2d\Vol_f.
\end{equation*}
The equality holds if and only if $|\nabla u | = cw$ for some real constant $c$.
\end{lemma}
\begin{proof}
Using the formula $(\ref{9})$ in Lemma \ref{25}, we have  for all test function $h$:
\begin{equation}\label{11}
\int_{\Sigma}\left(|\nabla^2 u|^2- |\nabla|\nabla u||^2+\Ric_f(\nabla u,\nabla u)\right)h^2\, d\Vol_f
\end{equation} 
$$
\hspace{2cm}\leq \int_{\Sigma}|\nabla h|^2|\nabla u|^2d\Vol_f-\int_{\Sigma}w^2\left|\nabla\left(\frac{h|\nabla u|}{w}\right)\right|^2d\Vol_f.
$$

Using the weighted Cauchy inequality for vectors we get
$$|X+Y|^2\geq |X|^2+|Y|^2-2|X||Y|\geq (1-\delta)|X|^2+(1-\delta^{-1})|Y|^2,$$
for each $\delta>0$, and so the last term of the right hand side can be arranged as

 $$
 w^2\left|\nabla\left(\frac{h|\nabla u|}{w}\right)\right|^2
\geq(1-\delta^{-1})|\nabla u|^2|\nabla h|^2+(1-\delta)h^2w^2\left|\nabla\left(\frac{|\nabla u|}{w}\right)\right|^2.
$$

Plugging this into $(\ref{11})$ yields

\begin{align*}
&\int_\Sigma\left(|\nabla^2 u|^2-|\nabla|\nabla u||^2 + Ric_f(\nabla u, \nabla u)\right)h^2d\Vol_f + \\  
&\hspace{0.5cm}(1-\delta)\int_\Sigma h^2w^2\left|\nabla\left(\frac{|\nabla u|}{w}\right)\right|^2d\Vol_f
\leq\frac{1}{\delta}\int_\Sigma|\nabla h|^2|\nabla u|^2d\Vol_f.
\end{align*}
Choosing $\delta=\frac{1}{2}$ we conclude the desired result.

\end{proof}

\section{Proof of theorem \ref{10} }

{\bf Proof of Theorem \ref{10}.} 
We claim that, for a suitable family $\{h_j\}_{j\in\mathbb{N}}$, it holds
\begin{equation}\label{14}
\hspace{0.75cm} \{h_{j}\}\,\,\text{is monotone increasing to 1,} \,\,\,\,\lim_{j\rightarrow+\infty}\int_\Sigma|\nabla h_{j}|^2|\nabla u|^2d\Vol_f=0.
\end{equation}
Choose $h_j$ as follows, according to the case.

In the first case, { (i)}, fix $\Omega\Subset \Sigma$ with smooth boundary and let $\{\Omega_j\}\uparrow \Sigma$ be a smooth exhaustion with 
$\Omega\subset \Omega_1$. Choose $h_j\in Lip_c(M)$ to be identity $1$ on $\Omega$, $0$ on $M\backslash\Omega_j$ and the  $f$-harmonic capacitor on $\Omega_j\backslash\Omega$, that is, the solution of

$$\begin{cases}
 \Delta_fh_j=0\,\,\,\,&\text{on } \Omega_j\backslash\Omega\\
 h_j=1\,\,\,\,& \text{on } \partial\Omega,\\
 h_j=0\,\,\, &\text{on } \partial \Omega_j,
\end{cases}
$$
which exists by variational arguments. By Maximum principle and since $\Sigma$ is $f$-parabolic, $\{h_j\}$ is monotonically increasing and pointwise convergent to $1$, and furthermore
$$\int_{\Omega_j}|\nabla\phi_j|^2|\nabla u|^2d\Vol_f\leq |\nabla u|^2_{L^{\infty}}\text{cap}_f(\Omega,\Omega_j)\rightarrow |\nabla u|^2_{L^{\infty}}\text{cap}_f(\Omega)=0,$$
where we have used the { Proposition \ref{13}}. This conclude $(\ref{14})$ under the hypothesis (i).

In the second case, { (ii)}, we apply a logarithmic cut-off argument. For fixed $R>0$, choose the following radial function $h(x)=h_R(r(x))$:
\begin{equation*}
 h_R(r)=\begin{cases}
  1 & \text{if } r\leq\sqrt{R},\\
  2-2\frac{\log r}{\log R} &\text{if } r\in[\sqrt{R},R],\\
  0          &\text{if } r\geq R.
 \end{cases}
\end{equation*}
Note that 
$$|\nabla h(x)|^2=\frac{4}{r(x)^2\log^2R}\chi_{B_R\backslash B_{\sqrt{R}}}(x),$$
where $\chi_A$ is the characteristic function of a subset $A\subseteq M$. Choose $R$ in such a way that $\log R/2$ is an integer. Then

\begin{align}\label{15}
 \int_M|\nabla h|^2|\nabla u|^2d\Vol_f&=\int_{B_R\backslash B_{\sqrt{R}}}|\nabla h|^2|\nabla 
u|^2d\Vol_f\\\nonumber
&=\frac{4}{\log^2R}\sum_{k=\log R/2}^{\log R-1}\int_{B_{e^{k+1}}\backslash B_{e^k}}\frac{|\nabla u|^2}{r(x)^2}d\Vol_f\\\nonumber
&\leq \frac{4}{\log^2R}\sum_{k=\log R/2}^{\log R}\frac{1}{e^{2k}}\int_{B_{e^{k+1}}}|\nabla u|^2d\Vol_f.
\end{align}

By assumption
$$\int_{B_{e^{k+1}}}|\nabla u|^2d\Vol_f\leq(k+1)e^{2(k+1)}\gamma(k),$$
for some $\gamma(k)$ satisfying $\gamma(k)\rightarrow 0$ as $k\rightarrow+\infty$. Without loss of generality, we can assume $\gamma(k)$  
to be decreasing as function of $k$. Hence,

 \begin{align}\label{16}
  \frac{4}{\log^2R}\sum_{k=\log R/2}^{\log R}\frac{1}{e^{2k}}\int_{B_{e^{k+1}}}|\nabla u|^2d\Vol_f&\leq \frac{8}{\log^2R}\sum_{k=\log 
R/2}^{\log R}\frac{e^{2(k+1)}}{e^{2k}}(k+1)\gamma(k)\\\nonumber
&\leq \frac{8e^2}{\log^2R}\gamma(\log R/2)\sum_{k=0}^{\log R}(k+1)\\\nonumber
&\leq\frac{C}{\log^2 R}\gamma(\log R/2)\log^2 R\\\nonumber
&=C\gamma(\log R/2),
 \end{align}
for some constant $C>0$. Combining $(\ref{15})$ and $(\ref{16})$ and letting $R\rightarrow+\infty$ we deduce $(\ref{14})$. Therefore, in both cases, from Lemma \ref{I1}  we infer 

$$\int_\Sigma\left(|\nabla^2 u|^2-|\nabla|\nabla u||^2 + Ric_f(\nabla u, \nabla u)\right)d\Vol_f \leq 0.$$

Now we are able to apply the Co-area formula and Lemma \ref{19}, and so we obtain

$$\int_\real\int_{\Sigma_t}\left(|A_t|^2 - \frac{H^2_t}{n-1}\right)|\nabla u|\, d\hf dt\leq  $$
$$\int_\Sigma\left(|\nabla^2 u|^2-|\nabla|\nabla u||^2 + Ric_f(\nabla u, \nabla u)\right)d\Vol_f \leq 0.$$

Since for all operator $T$ acting on a vector space of dimension $n$ we have that $|T|^2 \geq \frac{trace(T)^2}{n-1}$ and the equality holds if and only if $T$ is multiple of the identity, we obtain: 
\begin{equation}\label{17}
 |\nabla u|=cw,\,\,\,\text{ for some } c>0,\,\,\,\,|\nabla^2 u|^2=|\nabla|\nabla u||^2,\,\,\,\, \Sigma_t \, \,  \mbox{is totally umbilical.}
\end{equation}

Consider $\Phi$ the flow associated of $\nu=\nabla u/|\nabla u|$, which is well-defined on $\real\times\Sigma$ because $\Sigma$ is complete and $|\nu|=1$. By $(\ref{17})$ and { Lemma \ref{19}}, $|\nabla u|$ is constant on each connected component of a level set 
$\Sigma_t$. Therefore, in an adapted orthonormal frame $\{e_j,e_n=\nu\}$ for the level set $\Sigma_t$, we have that
\begin{equation} 
\left\{ \begin{array}{l}\label{20}
  |A_t|^2=\frac{H_t^2}{n-1}\,\,\,  \text{implies} \, \, \,  \nabla^2 u(e_i,e_j)=\lambda\delta_{ij},\\
  0=\langle\nabla|\nabla u|,e_j\rangle=\nabla^2 u(\nu,e_j).\\
 \end{array}\right.
\end{equation}
We point out that, by first equality in \ref{17}, $u$ has no critical points. Now on, we will  verify the property claimed in the Theorem.

Let $\gamma$ be any integral curve of $\nu$, we will prove that is a geodesic. Indeed, let $X\in\mathfrak{X}(\Sigma)$ 
be a vector field, we have that
\begin{align*}
\langle \nabla_{\gamma^{\prime}}\gamma^{\prime},X\rangle&=\frac{1}{|\nabla u|}\langle\nabla_{\nabla u}\left(\frac{\nabla u}{|\nabla u|}\right), X\rangle\\
&=\frac{1}{|\nabla u|^2}\langle\nabla_{\nabla u}\nabla u, X\rangle-\frac{1}{|\nabla u|^3}\langle\nabla u(|\nabla u|)\nabla u, X\rangle\\
&=\frac{1}{|\nabla u|^2}\Hess u(\nabla u,X)-\frac{1}{|\nabla u|^3}\langle\nabla|\nabla u|,\nabla u\rangle\langle\nabla u, X\rangle\\
&=\frac{1}{|\nabla u|}\Hess u(\nu, X)-\frac{1}{|\nabla u|}\langle\nabla |\nabla u|,\nu\rangle\langle\nu, X\rangle\\
&=\frac{1}{|\nabla u|}\Hess u(\nu, X) - \frac{1}{|\nabla u|}\Hess u(\nu,\nu)\langle\nu, X\rangle=0,
\end{align*}
where we have used $(\ref{20})$. So, $ \nabla_{\gamma^{\prime}}\gamma^{\prime}=0$ and $\gamma$ is a geodesic as claimed.
\medskip

Following the arguments in the proof of \cite[Theorem 9.3]{pigola}, step-by-step
we will prove the topological splitting. Since $|\nabla u|$ is constant 
on level sets of $u$, $|\nabla u|=\beta(u)$ for some function $\beta$. 
Evaluating $u$ along curves $\Phi_t(x)$, since $u\circ\Phi_t$ is a local 
bijection, we deduce that $\beta$ is continuous. 
\begin{claim} 
$\Phi_t$ moves level sets of $u$ to level sets of $u$.
\end{claim}
Indeed, integrating $\frac{d}{ds}(u\circ\Phi_s)=|\nabla u|\circ\Phi_s=\beta(u\circ\Phi_s)$ 
we get 
$$t=\int_{u(x)}^{u(\Phi_t(x))}\frac{d\xi}{\beta(\xi)},$$
thus $u(\Phi_t(x))$ is independent of $x$ varying in a level set. As $\beta(\xi)>0$,
this also show that flow lines starting from a
level set of $u$ do not touch the same level set and we conclude the claim.\qed

\medskip
Let $N$ be a connected component of a level set $\Sigma_t$ of $u$. 
\begin{claim}
$\Phi|_{\real\times N}$ is surjective.
\end{claim}
Indeed,  since the flow of $\nu$ is through geodesics, for each $x\in N$, $\Phi_t$ 
coincides with the normal exponential map $\exp^{\perp}(t\nu(x))$. Moreover, 
since $N$ is closed in $\Sigma$ and $\Sigma$ is complete, the normal exponential map 
is surjective because each geodesic from $x\in \Sigma$ to $N$ minimizing distance dist$(x,N)$
is perpendicular to $N$ (by variational arguments). \qed

\begin{claim}
$\Phi|_{\real\times N}$ is injective.
\end{claim}
Suppose that $\Phi(t_1, x_1)=\Phi(t_2, x_2)$. Then, since 
$\Phi$ moves level sets to level sets, necessarily $t_1=t_2=t$. 
If by contradiction $x_1\neq x_2$, two distinct flow lines of $\Phi_t$ would
intersect at the point  $\Phi_t(x_1)=\Phi_t(x_2)$, contradicting the fact that 
$\Phi_t$ is a diffeomorphism on $\Sigma$ for every $t$, as desired. \qed

\medskip
From the former arguments,  we conclude that  $\Phi:\mathbb R\times N\rightarrow \Sigma$ is a
 diffeomorphism. In particular, each level set $\Phi_t(N)$ is connected. This proves the 
topological part of the splitting. 

To conclude the proof, we shall verify that $\Phi$ is isometry whether we consider the product space, $\real\times N$, endowed with a warped metric with warping function depends of $\nabla u$. 
Indeed, we consider  the
Lie derivative of the metric in the direction of $\nu$:
\begin{align*}
 \hspace{-0.15cm}(\mathcal{L}_{\nu}g_\Sigma)(X,Y)&=\langle\nabla_X\nu,Y\rangle+\langle X,\nabla_Y\nu\rangle\\
 &=\frac{2}{|\nabla u|}\nabla^2 u(X,Y)+X\left(\frac{1}{|\nabla u|}\right)\langle\nabla u,Y\rangle+
 Y\left(\frac{1}{|\nabla u|}\right)\langle\nabla u,X\rangle.
\end{align*}
So, using that $|\nabla u|$ is constant on $N$ and the properties of $\nabla^2 u$, we obtain that
$$
(\mathcal{L}_{\nu}g_\Sigma)(X,Y)=\frac{2}{|\nabla u|}\nabla^2 u(X,Y)=\left\{\begin{array}{rcl}
 -2\lambda g_\Sigma(X,Y),& if&  X,Y\in T\Sigma_t \\
  0,& if & X\,  \mbox{or} \,  Y \mbox{are not in} \, T\Sigma_t
\end{array}.\right.
$$
If, however,  $X$ and $Y$ are normal (w.l.o.g. take $X=Y=\nabla u$), we have
\begin{align*}
 (\mathcal{L}_{\nu}g_\Sigma)(X,Y)&=\frac{2}{|\nabla u|}\nabla^2 u(\nabla u,\nabla u)+2\nabla u\left(\frac{1}{|\nabla u|}\right)|\nabla u|^2\\
 &=\frac{2}{|\nabla u|}\nabla^2 u(\nabla u,\nabla u)-2\nabla u(|\nabla u|)\\
 &=\frac{2}{|\nabla u|}\nabla u(|\nabla u|^2)-2\langle\nabla|\nabla u|,\nabla u\rangle=0.
\end{align*}
Summarizing, we get
$$(\mathcal{L}_\nu g_\Sigma(X,Y) = -2\lambda(\langle X,Y\rangle - (\nu\otimes\nu)(X,Y)),$$
and $\lambda= - \frac{1}{n-1}div_{\Sigma}(\frac{\nabla u}{|\nabla u|})$.
\medskip

Thus, it is a standard computation to verify that $\Phi$ is isometry from $\real\times N$ endowed with the warped metric  $ds^2 = dt^2 + e^{-2\int_0^t\lambda(s)\, ds} g_N$ to $\Sigma$.
{\flushright \qed}

\section{Proof of Theorem \ref{cor1}}

{\bf Proof of Theorem \ref{cor1}.} Under this assumption, the level sets of the function $u$ are totally geodesic and so $\lambda$ obtained in former result vanishes. Thus we conclude that $\Sigma$ splits as a Riemannian product, as desired. In particular, the weighted Ricci curvature of $N$ is nonnegative.
%if $n\geq 3$,  while, if $n=2,\, \Sigma=\mathbb R^2$ or $\real\times \mathbb S^1$ with the 
%flat metric.

Lastly, we shall verify  the properties of the function $u$.  Let $\gamma$ be any 
integral curve of $\nu$. Then
$$\frac{d}{dt}(u\circ\gamma)=\langle\nabla u,\nu\rangle=|\nabla u|\circ\gamma>0,$$
since $|\nabla u|>0$. As $\Sigma$ splits 
isometrically in the direction of $\nabla u$, we get that $\Ric(\nu,\nu)=0$
and this imply  that $\Hess f(\nu,\nu)=0$. Consequently
$\langle\nabla f,\nu\rangle$ is a constant $k$ (depending of $f$)  in the splitting direction. 

By the other hand,
\begin{align*}
 -g(u\circ\gamma)=\Delta_fu(\gamma)&=\Hess u(\nu,\nu)(\gamma)-\langle\nabla f,\nabla u\rangle(\gamma)\\
 &=\langle\nabla|\nabla u|,\nu\rangle(\gamma)-\langle\nabla f,\nu\rangle|\nabla u|(\gamma)\\
 &=\frac{d}{dt}(|\nabla u|\circ\gamma)-k|\nabla u|\circ(\gamma)\\
 &=\frac{d^2}{dt^2}(u\circ\gamma)-k\frac{d}{dt}(u\circ\gamma),
\end{align*}
and thus $y=u\circ\gamma$ solves the ODE $-y^{\prime\prime}+ky^{\prime}=g(y)$ with $y^{\prime}>0$.

We next address the $f$-parabolicity. Under assumption ${\text{(i)}}$, $\Sigma$ is $f$-parabolic and so $N$ is necessarily $f$-parabolic too. We are 
going to deduce the same under assumption  ${\text{(ii)}}$. Note that the chain of inequalities
\begin{align*}
\left(\int_{-R}^R|y^{\prime}(t)|^2dt\right)\Vol_f(B_R^N)&\leq\int_{[-R,R]\times B_R^N}|y^{\prime}(t)|^2dt\, d\Vol_f^N\\
&\leq\int_{B_{R\sqrt{2}}}|\nabla u|^2d\Vol_f=o(R^2\log R)
\end{align*}
gives immediately the desired result, since $|y^{\prime}|>0$ everywhere. Thus, since $\Vol_f(B_R^N)=o(R^2\log R)$, we know that 
there is a constant $A$ such that $\Vol_f(B_R^N)\leq AR^2\log R$, that is,  
$$\frac{R}{\Vol_f(B_R^N)}\geq \frac{1}{AR\log R},$$
and so
\begin{align*}
 \lim_{t\to\infty}\int_1^{t}\frac{R\, dR}{\Vol_f(B_R^N)}&\geq A^{-1}\lim_{t\to\infty}\int_1^{t}\frac{dR}{R\log R}=A^{-1}\lim_{t\rightarrow\infty}\log(\log t)=\infty,
\end{align*}
and thus, by { Proposition \ref{23}},  $N$ is $f$-parabolic.

{\flushright \qed}

\end{document}